\documentclass[11pt,reqno,a4paper]{amsart}
\usepackage{graphicx,color}

\usepackage{amsmath,amsfonts,amssymb,amsthm,mathrsfs,amsopn}
\usepackage{latexsym}

\usepackage[numbers]{natbib}

\usepackage{hyperref}

\usepackage[usenames,dvipsnames,x11names,table]{xcolor}

\allowdisplaybreaks

\numberwithin{equation}{section}

\def\H{\mathcal H}

\def\R{\mathbb R}
\def\N{\mathbb N}

\newcommand{\dist}{\mathop{\mathrm{dist}}}
\def\e{\varepsilon}
\def\s{\sigma}

\def\vphi{\varphi}

\def\g{\gamma}

\def\Id{{\rm Id}}

\def\spt{{\rm spt}}

\def\pa{\partial}

\def\00{{\bf 0}}

\def\F{\mathcal{F}}

\def\P{\mathcal{P}}

\def\CC{\mathcal{C}}

\newcommand{\D}{\Delta}

\newcommand{\cc}{\subset\subset}

\def\Lip{{\rm Lip}\,}

\newcommand\res{\mathop{\hbox{\vrule height 7pt width .3pt depth 0pt \vrule height .3pt width 5pt depth 0pt}}\nolimits}

\def\weak{\stackrel{*}{\rightharpoonup}}

\def\FF{\mathbf{F}}

\def\D{\mathsf{D}}

\def\Id{\mathrm{Id}}

\theoremstyle{plain}
\newtheorem{theorem}{Theorem} [section]

\newtheorem{lemma}[theorem]{Lemma}
\newtheorem{proposition}[theorem]{Proposition}
\theoremstyle{definition}
\newtheorem{definition}[theorem]{Definition}
\newtheorem{remark}[theorem]{Remark}
\newtheorem*{ack}{Acknowledgements}

\title{A direct approach to the anisotropic Plateau problem}

\author{C. De Lellis}
\address{Institut f\"ur Mathematik, Universitaet Z\"urich, Winterthurerstrasse 190, CH-8057 Z\"urich, Switzerland}
\email{camillo.delellis@math.uzh.ch}

\author{A. De Rosa}
\address{Institut f\"ur Mathematik, Universitaet Z\"urich, Winterthurerstrasse 190, CH-8057 Z\"urich, Switzerland}
\email{antonio.derosa@math.uzh.ch}

\author{F. Ghiraldin}
\address{Institut f\"ur Mathematik, Universitaet Z\"urich, Winterthurerstrasse 190, CH-8057 Z\"urich, Switzerland. Max Planck Institut for Mathematics in the Sciences, Inselstrasse 22, 04103 Leipzig, Germany}
\email{francesco.ghiraldin@math.uzh.ch, Francesco.Ghiraldin@mis.mpg.de}

\begin{document}

\begin{abstract}
We prove a compactness principle for the anisotropic formulation of the Plateau problem in codimension one, along the same lines of previous works of the authors \cite{DelGhiMag,DePDeRGhi}. In particular, we perform a new strategy for proving the rectifiability of the minimal set, avoiding the Preiss' Rectifiability Theorem \cite{Preiss}.
\end{abstract}

\maketitle

\section{Introduction}

The anisotropic Plateau's problem aims at finding an energy minimizing surface 
spanning a given boundary when the energy functional is more general than the usual surface area (as in the standard Plateau's problem) and is obtained integrating a general Lagrangian $F$ over the surface. In particular, the integrand depends on the position and the tangent space to the surface.

As in the case of the area integrand, 
\cite{DeGiorgiSOFP1,federer60,reifenberg1,Almgren68,DavidSemmes,HarrisonPugh14,DelGhiMag,DePDeRGhi}, many definitions of boundary conditions (both homological and homotopical), 
as well as the type of competitors (currents, varifolds, sets) have been considered in the literature. 
An important existence, regularity and almost uniqueness result in arbitrary dimension and codimension was achieved by Almgren in \cite{Almgren68}, using refined techniques from geometric measure theory. In more recent times, the Plateau problem for the area integrand has been investigated in order to give an existence theory which could comprehend several notions of competitors all together \cite{Feuvrier2009,HarrisonPugh14,davidplateau, DelGhiMag,DePDeRGhi}. In particular a very elegant notion of boundary for general closed sets has been introduced by Harrison and Pugh in \cite{HarrisonPugh14}, which proves the existence and regularity for minimizers of the area functional in codimension 1. The same authors in
 \cite{HarrisonPugh15} investigated the ``inhomogeneous Plateau's problem'', where the energy density is isotropic but depends on the space position of the surface, proving existence of a minimizer under a suitable cohomological definition of boundary. 

In this paper, we adopt the same strategy as in \cite{DelGhiMag, DePDeRGhi}, namely we prove a general compactness theorem for minimizing sequences in general classes of rectifiable sets. More precisely, we consider the measures naturally associated to any such sequence and we show that, if a sufficiently large class of deformations are admitted, any weak limit is induced by a rectifiable set, thus providing compactness and semicontinuity under very little assumptions. Our result does not address directly the question whether the limiting set belongs to the original class, which is linked to its closure under weak convergence of measures. However, we can easily show that this is the case for the class considered by Harrison and Pugh in \cite{HarrisonPugh14}, thus giving a generalization of their existence theorem to any (elliptic) anisotropic functional. While we were completing this paper we learned that analogous results have been obtained at the same time by Harrison and Pugh in \cite{HarrisonPugh16}, using different arguments and building upon their previous work \cite{HarrisonPugh15}.

One main difficulty in our approach is to prove the rectifiability of the support of the limiting measure. 
In the paper \cite{DelGhiMag}, the key ingredient to obtain such rectifiability is the classical monotonicity formula for the mass ratio of the limiting measure, which allows to apply Preiss' rectifiability theorem for Radon measures \cite{Preiss,DeLellisNOTES}. 
Such a strategy does not seem feasible for general anisotropic integrands, where the monotonicity of the mass ratio is unlikely to be true, as pointed out in \cite{Allardratio}. Since all the other ingredients of \cite{DelGhiMag} can be easily transported to the anisotropic case, the main goal of this paper is to show how, in codimension one, the rectifiability of the limiting measure follows from the theory of Caccioppoli sets, bypassing the monotonicity formula and the deep result of Preiss.
In particular, we are able to prove the results analogous to those of \cite{DelGhiMag} with a strategy which has some similarities with the one used in \cite{Almgren68}. 

In \cite{DePDeRGhi}, a similar theorem for the area functional was proved in any codimension. The most general case of any codimension and anisotropic energies will be addressed in a further paper by the same authors, see \cite{DePDeRGhi2}, using however different and more sophisticated PDE techniques~\cite{DePDeRGhiCPAM}.
 
The structure of this note is the following: in Section~\ref{s2} we introduce the notation and state the main theorems of the paper. 
 Section~\ref{s3} is devoted to the proof of the main compactness result. In Section~\ref{s4} we analyse the applicability of Theorem~\ref{thm generale} for some specific boundary conditions and investigate the regularity of minimizers.
\medskip
\begin{ack}
This work has been supported by ERC 306247 {\it Regularity of area-minimi\-zing currents} and by
SNF 146349 {\it Calculus of variations and fluid dynamics}. 
\end{ack}

\section{Notation and main results}\label{s2}
The ambient space is the standard euclidean one, \(\R^{n+1}\), 
and $\H^k$ denotes the $k$-Hausdorff measure; 
moreover, for every set $A$, we let $|A|:=\H^{n+1}(A)$ be its Lebesgue measure. 
We will let $U_r(A)$ be the open tubular neighborhood of $A$ of radius $r$. 
Recall that a set \(K\) is said to be \(n\)-rectifiable if it can be covered, 
up to an \(\H^n\) negligible set, by countably many \(C^1\) $n$-dimensional 
submanifolds, see \cite[Chapter 3]{SimonLN}; we also denote by
$G=G(n+1,n)$ the Grassmannian of unoriented $n$-dimensional 
hyperplanes in $\R^{n+1}$. Given an $n$-rectifiable set $K$, 
we denote by $T_K(x)$ the approximate tangent space of $K$ at $x$, 
which exists for $\H^n$-almost every point $x \in K$ \cite[Chapter 3]{SimonLN}. 
Finally, we let (for $x\in \mathbb R^{n+1}$, $0<r<\infty$, $0<a<\infty$ and $0<b<\infty$)
\begin{itemize}
\item $B_{x,r}:=\{y\in\R^{n+1}:|x-y|<r\}$; 
\item $B:=B_{0,1}$;
\item $Q_r := ]-\frac{r}{2}, \frac{r}{2}[^{n+1}$; 
\item $R_{2a,2b}:=[-a,a]^n\times [-b,b]$;
\item $\omega_n := \H^n(B\cap(\R^n\times\{0\}))$ and $\sigma_n:=\H^n(\pa B)$. 
\end{itemize}

The anisotropic Lagrangians considered in the rest of the note will be continuous maps 
$$
F: \R^{n+1}\times G\ni (x,\pi)\mapsto F(x,\pi) \in  \R^+ = ]0, \infty[ ,
$$
verifying the lower and upper bounds 
\begin{equation}\label{cost per area}
0 < \lambda \leq F(x,\pi) \leq \Lambda<\infty \qquad \forall (x,\pi)\in \R^{n+1}\times G.
\end{equation}
Given an $n$-rectifiable set $K$ and an open subset $U\subset \R^{n+1}$, we define:
\begin{equation}\label{energia}
\FF(K,U) := \int_{K\cap U} F(x,T_K(x))\, d\H^n(x) \mbox{ \ \ and \ \ } \FF(K) := \FF(K,\R^{n+1}).
\end{equation}
It will be also convenient to look at the frozen Lagrangian: for $y\in\R^{n+1}$, we let
\begin{equation*}
\FF^y(K,U) := \int_{K\cap U} F(y,T_K(x))\, d\H^n(x).
\end{equation*}

Throughout all the paper, $H\subset \mathbb R^{n+1}$ will denote a closed subset of $\R^{n+1}$. 
Assume to have a class $\mathcal{P} (H)$ of relatively closed $n$-rectifiable subsets $K$ of $\mathbb R^{n+1}\setminus H$: 
one can then formulate the anisotropic Plateau's problem by asking whether the infimum
\begin{equation}
  \label{plateau problem generale}
m_0 :=  \inf \big\{\FF(K) : K\in \mathcal{P}(H)\big \}
\end{equation}
is achieved by some set (which is the limit of a minimizing sequence), if it belongs to the chosen class $\P(H)$ and which additional  regularity properties  it satisfies. 

We next outline a set of flexible and rather weak requirements for $\mathcal{P}(H)$. 

\begin{definition}[Cup competitors]\label{def:cup}
Let $K\subset \mathbb R^{n+1}\setminus H$ and $B_{x,r}\subset \subset \mathbb R^{n+1}\setminus H$.
We introduce the following equivalence relation among points of $\overline{B_{x,r}}\setminus K$: 
 \[
 y_0\sim_{K,x,r} y_1 \quad \Longleftrightarrow\quad \exists \, \gamma\in C^0([0,1],\overline{B_{x,r}}\setminus K) :\, \gamma(0)=y_0, \, \gamma(1) = y_1,\, \gamma(]0,1[)\subset B_{x,r}
 \]
(where $x$ and $r$ are clear from the context we will omit them and simply write $\sim_{K}$). 
We enumerate as \{$\Gamma_i(K, x,r)\}$ the equivalence classes in $\partial B_{x,r} /\sim_{K,x,r}$ (where the index $i$ varies either among all natural numbers or belongs to a finite subset of them). 
The cup competitor associated to $\Gamma_i(x,r)$ for $K$ in $B_{x,r}$ is
\begin{equation}
  \label{cup comp}
  \big(K\setminus B_{x,r}\big) \cup \big((\partial B_{x,r}) \setminus \Gamma_i(K, x,r)\big)\,.
\end{equation}
\end{definition}
For further reference we also introduce the sets
 \begin{equation}\label{e:Omega}
\Omega_i(K, x,r)  =\{ z \in B_{x,r}\setminus K : \exists \, y \in \Gamma_i(K, x,r)  \mbox{ such that } z\sim_{K,x,r} y \}\, .
\end{equation}
The dependence on $K$, $x$ and $r$ will be sometimes suppressed if clear from the context. 
It is easy to see that the associated sets $\Omega_i(K, x,r)$ are connected components of $B_{x,r}\setminus K$ (possibly not all of them). 

\begin{definition}[Good class]\label{def good class}
A family $\mathcal{P} (\mathbf{F}, H)$ of relatively closed subsets $K\subset \mathbb R^{n+1}\setminus H$ is called a {\em good class} if 
 for any $K\in \mathcal{P} (\mathbf{F}, H)$, for every $x\in K$ and for a.e. $r\in (0, \dist (x, H))$ the following holds:
\begin{equation}
  \label{inf good class}
  \inf \big \{ \FF (J) : J \in \mathcal{P} (H)\,,J\setminus \overline{B_{x,r}} = K\setminus \overline {B_{x,r}}\big \} \leq \FF (L)\,
\end{equation}
whenever $L$ is any cup competitor for $K$ in $B_{x,r}$.
\end{definition}

\begin{remark}\label{r:compare_cup}
Observe that the definition of cup competitors is a slight modification of that of \cite{DelGhiMag}, where $\Gamma_i (K, x,r)$ were taken to be connected components of $\partial B_{x,r}\setminus K$: observe however that, for every cup competitor in \cite{DelGhiMag}, we can find a cup competitor as above which has at most the same area, since each $\Gamma_i (K,x,r)$ is a union of connected components of $\partial B_{x,r}\setminus K$ and each connected component of $\partial B_{x,r}\setminus K$ is contained in at least one $\Gamma_i (K,x,r)$. Finally good classes in this paper do not assume any kind of comparisons with cones, as it is the case of \cite{DelGhiMag}.
\end{remark}

The point of our note is that the notion of good class is enough to ensure that any weak$^*$ limit of a minimizing sequence is a rectifiable measure and that a suitable lower semicontinuity statement holds for energies $\FF$ which satisfy the usual ellipticity condition of \cite[5.1.2]{FedererBOOK}, cf. Theorem \ref{thm generale} below. In particular, as shown in \cite[5.1.3-5.1.5]{FedererBOOK}, the convexity of the integrand $F$ is a sufficient condition and we take it therefore as definition here. 

\begin{definition}[Elliptic anisotropy, {\cite[5.1.2-5.1.5]{FedererBOOK}}]
$F$ is elliptic if its even and positively $1$-homogeneous extension to $\mathbb R^{n+1}\times (\Lambda_n (\mathbb R^{n+1})\setminus \{0\})$ is $C^2$ and it is uniformly convex in the $\pi$ variable on compact sets. 
\end{definition}

Actually the only points required in the proof of Theorem \ref{thm generale} are the lower semicontinuity of the functional $\mathbf{F}$ under the usual weak convergence of reduced boundaries of Caccioppoli sets and the following estimate on the oscillation of $F$ over 
compact sets $V\subset\subset\R^{n+1}$:
\begin{equation}\label{oscill}
\sup_{x,y\in V\, S,T\in G(n+1,n)}|F(x,T)-F(y,S)|\leq \omega_V (|x-y|+\|T-S\|),
\end{equation}
where $\omega_V$ is a modulus of continuity which depends upon $V\times G(n+1,n)$ and $\|\cdot\|$ is the standard metric on $G (n,n+1)$ defined as in \cite[Chapter 8, Section 38]{SimonLN}.
In particular the $C^2$ regularity of the definition above can be considerably relaxed.

We now have all the tools to state our main theorem. A minimizing sequence $\{K_j\} \subset \mathcal{P} (\mathbf{F}, H)$ in Problem \eqref{plateau problem generale} satisfies the property $\FF(K_j) \to m_0$, and throughout the paper we will assume $m_0$ to be finite. 

\begin{theorem}\label{thm generale}
Let $H\subset \mathbb R^{n+1}$ be closed and $\mathcal{P} (\mathbf{F}, H)$ be a good class. 
Let $\{K_j\}\subset \mathcal{P}(\mathbf{F}, H)$ be a minimizing sequence and assume $m_0 < \infty$. Then, up to subsequences, the measures $\mu_j := F(\cdot,T_{K_j}(\cdot))\H^n \res K_j$ converge weakly$^\star$ in $\mathbb R^{n+1}\setminus H$ to a measure $\mu = \theta \H^n \res K$, where $K = \spt\, \mu$ is an  $n$-rectifiable set and  $\theta \geq c_0$ for some constant $c_0 (\mathbf{F}, n)$. 

Moreover, if $\mathbf{F}$ is elliptic, then $\liminf_j\FF(K_j)\ge \FF(K)$ (that is $\theta (x)\geq F (x, T_K (x))$) and
in particular, if $K \in \mathcal{P} (\mathbf{F}, H)$, then $K$ is a minimum for Problem \eqref{plateau problem generale} and thus $\theta (x) = F (x, T_K (x))$.
\end{theorem}

Indeed the measure $\mu$ above is an $n$-dimensional rectifiable varifold in the sense of \cite[Chapter 4]{SimonLN}. 
Since the proof of Theorem \ref{thm generale} does not exploit Preiss' rectifiability Theorem, when the Lagrangian is constant (i.e. up to a factor it is the area functional $F \equiv 1$) and we require the stronger energetic inequality in \eqref{inf good class} to hold for any cup competitors as in \eqref{cup comp} \cite[Equation 1.2]{DelGhiMag},  then the same strategy gives a simpler proof of the conclusions of \cite[Theorem 2]{DelGhiMag}, except for the monotonicity formula in \cite[Equation (1.5)]{DelGhiMag}.

One application  of Theorem \ref{thm generale} yields a generalization of the main result in \cite{HarrisonPugh14}
to anisotropic Lagrangians. More precisely, consider the following classes of sets.

\begin{definition}
Let $n\ge 2$ and $H$ be a closed set in $\R^{n+1}$. Let us consider the family
\[
\CC_H=\big\{\g:S^1\to\R^{n+1}\setminus H:\mbox{$\g$ is a smooth embedding of $S^1$ into $\R^{n+1}$}\big\}\,.
\]
We say that $\CC\subset\CC_H$ is closed by homotopy (with respect to $H$) if $\CC$ contains all elements
$\gamma'\in \CC_H$ belonging to the same homotopy class $[\g] \in\pi_1(\R^{n+1}\setminus H)$ of any $\gamma \in \CC$. Given $\CC\subset\CC_H$ closed by homotopy, we denote by $\mathcal{F} (H, \CC)$ the  family of relatively closed subsets $K$ of $\R^{n+1}\setminus H$ such that
\begin{eqnarray*}
\mbox{$K\cap\g\ne\emptyset$ for every $\g\in\CC$}\,.
\end{eqnarray*}
\end{definition}

\begin{theorem}\label{thm plateau}
Let $n\ge 2$ and $\CC$ be closed by homotopy with respect to $H$. Let also $\P(H)=\{K \in \F(H,\CC) \, : \, K \mbox{ is $n$- rectifiable}\}$. Then:
\begin{itemize}
\item[(a)]  $\F (H,\CC)$ is a good class in the sense of Definition \ref{def good class}  for any functional $\mathbf{F}$. 
\item[(b)]  If $\{K_j\}\subset \P (H)$ is a minimizing sequence and  $K$  is any set associated to $\{K_j\}$ by Theorem \ref{thm generale}, then $K\in\P (H)$ and thus $K$ is a minimizer.
\item[(c)]  The set $K$ in (b) is an $(\FF , 0, \infty)$-minimal set in $\R^{n+1}\setminus H$ in the sense of Almgren {\rm\cite{Almgren76}}.
\end{itemize}
\end{theorem}

As already mentioned, a similar theorem has been obtained independently by Harrison and Pugh in \cite{HarrisonPugh16},
building upon a previous paper, \cite{HarrisonPugh15}, where the same authors considered the special case of isotropic Lagrangians 
$F (x, \pi) = f(x)$. 

Finally, we remark that it is possible to obtain the useful additional information $\theta (x) = F (x, T_K (x))$ in Theorem \ref{thm generale} even when we cannot directly infer that $K= \spt\, \mu$ belongs to the class $\mathcal{P} (\mathbf{F}, H)$, provided we allow a larger class of competitors.  We recall here the ones introduced in \cite{DePDeRGhi}.
\begin{definition}[Lipschitz deformations]%\label{d:deform}
Let $\mathfrak D(x,r)$ be the set of functions $\vphi:\R^{n+1} \rightarrow \R^{n+1}$ for which there exists a $C^1$ isotopy $\lambda:[0,1]\times \R^{n+1}\rightarrow\R^{n+1}$ such that 
$$
\lambda(0,\cdot) = \Id, \quad \lambda(1,\cdot)=\vphi, \quad \lambda(t,h)=h \quad\forall\,(t,h)\in [0,1]\times (\R^{n+1} \setminus B_{x,r}) .
$$
We finally set $\D(x,r):=\overline{\mathfrak D(x,r)}^{C^0}\cap \Lip (\mathbb R^{n+1})$, the sequential closure of $\mathfrak D(x,r)$ with respect to the uniform convergence, intersected with the space of Lipschitz maps.
\end{definition}

\begin{definition}[Deformed competitors and deformation class]\label{def deformed class}
Let  $K\subset \mathbb R^{n+1}\setminus H$ be relatively closed and  $B_{x,r}\subset\subset \mathbb R^{n+1}\setminus H$.
A  {\em deformed competitor} for $K$ in $B_{x,r}$ is any set of the form
\begin{equation*}
 \varphi \left (  K \right ) \quad \mbox{ where } \quad \varphi \in \D(x,r).
\end{equation*}
A family $\mathcal{P} (\mathbf{F}, H)$ of relatively closed \(n\)-rectifiable subsets $K\subset \mathbb R^{n+1}\setminus H$ 
is called a {\em deformation class} if for every $K\in\mathcal{P} (\mathbf{F}, H)$, for every $x\in K$ 
and for a.e. $r\in (0, \dist (x, H))$
\begin{equation}\label{inf def class}
  \inf \big\{ \FF (J) : J\in \mathcal{P} (H)\,,J\setminus \overline{B_{x,r}} =K\setminus \overline
{B_{x,r}} \big\} \leq \FF (L)\,
\end{equation}
whenever $L$ is any deformed competitor for $K$ in $B_{x,r}$.
\end{definition}

\begin{proposition}\label{p:aggiuntiva}
Assume that $\mathbf{F}$ is elliptic and that $H$, $\mathcal{P} (\mathbf{F}, H)$, $\{K_j\}$, $\mu$ and $K$ are as in Theorem \ref{thm generale}. If in addition $\mathcal{P} (\mathbf{F}, H)$ is a deformation class, then $\theta (x) = F (x, T_K (x))$ for $\mathcal{H}^n$-a.e. $x\in K$. 
\end{proposition}

\section{Proof of Theorem \ref{thm generale}}\label{s3}

Parts of the proofs follow the isotropic case treated in \cite{DelGhiMag}: we will be brief on these arguments, hoping to convey the main ideas and in order to leave space to the original content.
\medskip

The proof of Theorem \ref{thm generale} goes as follows: we consider the natural measures
$(\mu_j)$ associated to a minimizing sequence $(K_j)$ and extract a weak limit $\mu$. We first recall that, as a consequence of minimality, $\mu$ enjoys density upper and lower 
 bounds on $\spt(\mu)$, leading to the representation $\mu = \theta\H^n\res\spt(\mu)$: this part follows almost verbatim the proof of \cite{DelGhiMag}. 
Then, via an energy comparison argument, we exclude the presence of purely unrectifiable subsets of $\spt(\mu)$, which is
the core novelty of the note. 
We then show that, if the Lagrangian is elliptic,  then the energy is lower semicontinuous along $(K_j)$. Finally, if we assume also that $\mathcal{P} (\mathbf{F}, H)$ is a deformation class, we show that $\theta (x) = F(x, T_K (x))$.

\subsection{Density bound} In this section we prove the following

\begin{lemma}[Density bounds]\label{l:density}
Suppose that $\P(H)$ is a good class, that $\{K_j\}\subset\P(\mathbf{F}, H)$ is a minimizing sequence for problem \eqref{plateau problem generale} and that
$$
\mu_j = F(\cdot, T_K(\cdot))\H^n\res K_j \weak \mu
$$
in $\R^{n+1}\setminus H$. Then the limit measure $\mu$ enjoys density upper and lower bounds: 
\begin{equation}
  \label{lower density estimate mu}
  \theta_0\,\omega_n r^n\le \mu(B_{x,r}) \le \theta_0^{-1}\,\omega_n r^n\,,\qquad\forall x\in\spt\,\mu\,,\,\forall r<d_x :=\dist(x,H) \, 
\end{equation}
for some positive constant $\theta_0=\theta_0(n, F)>0$. 
\end{lemma}
\begin{proof}
The density lower bound can be proved as in \cite[Theorem 2, Step 1]{DelGhiMag} with the use of cup competitors only, since the energy $\FF$ is comparable to the Hausdorff measure by \eqref{cost per area}. The notion of cup competitor in Definition \ref{def good class} slightly differs from the notion in \cite[Definition 1]{DelGhiMag}, however the key fact is that the latter have larger energy, cf. Remark \ref{r:compare_cup}. 
The existence of a density upper bound is trivially true, since we can use a generic sequence $\{\Gamma_j\}$ of cup competitors associated to $\{K_j\}$ in $B_{x,r}$. Observe that at least one $\Gamma_j$ exists as long as $\partial B_{x,r}\setminus K_j\neq \emptyset$: on the other hand for a.e. radius $r$ we have $\liminf_j \mathcal{H}^{n-1} (K_j)< \infty$ and we can assume the existence of a subsequence for which the $\Gamma_j$ exist.  Hence, by almost minimality
\begin{equation} \label{s}
\mu (B_{x,r})\le\limsup_j \mu_j (\overline{B_{x,r}}) \leq \limsup_j \FF (\pa B_{x,r}\setminus \Gamma_j) \leq  \limsup_j \Lambda \, \H^n(\pa B_{x,r}\setminus \Gamma_j) \leq \Lambda \, \sigma_n \, r^n\,.
\end{equation}
\end{proof}
We remark that, if the requirement of being a good class were substituted by that of being a deformation class, the density lower bound could be proven as in \cite[Theorem 1.3, Step 1]{DePDeRGhi}: note that although the bound in \cite[Theorem 1.3, Step 1]{DePDeRGhi} is claimed for the area functional, the argument requires only the two-sided comparison of \eqref{cost per area}. Moreover, the upper bound could be obtained as in \eqref{s}, but using the slightly different cup competitors defined in \cite[Definition 1]{DelGhiMag}, which are proven to be deformed competitors in \cite[Theorem 7, Step 1]{DelGhiMag}.

\subsection{Proof Theorem \ref{thm generale}: rectifiability}
Up to extracting subsequences, we can assume the existence of a Radon measure $\mu$ on $\R^{n+1}\setminus H$ such that
\begin{equation}\label{muj va a mu}
  \mu_j\weak\mu\,,\qquad\mbox{as Radon measures on $\R^{n+1}\setminus H$}\,.
\end{equation}
We set $K = \spt\, \mu \setminus H$ and from the differentiation theorem for Radon measures, see for instance \cite[Theorem 6.9]{mattila}, and Lemma \ref{l:density} we deduce
\begin{equation}\label{3}
\mu = \theta \H^n \res K, 
\end{equation}
is a relatively closed set in $\R^{n+1}\setminus H$ and $\theta: K \to \mathbb R^+$ a Borel function with $c_0\leq \theta\leq C_0$. 

We decompose $K = \mathcal R \cup \mathcal N$ into a rectifiable $\mathcal R$ and a purely unrectifiable $\mathcal N$ (see \cite[Chapter 3, Section 13.1]{SimonLN}) and assume by contradiction that $\H^n(\mathcal N)>0$.
Then, there is $x \in K$ such that
\begin{equation}
\label{nulldens}
  \Theta^{n}(\mathcal R,x)=\lim_{r\rightarrow 0}\frac{\H^n(\mathcal R\cap B_{x,r})}{\omega_n r^n} =0,\qquad   \Theta^{*n}(\mathcal N,x)=a>0. 
\end{equation}
Without loss of generality, we assume that $x=0$. The overall aim is to show that at $0$ the density lower bound of Lemma \ref{l:density} would be false, reaching therefore a contradiction. 

For every $\rho>0$, we let $\Omega_i(\rho)$, with $i \in \N$, be sets of \eqref{e:Omega} (where we omit the dependence on $K$ and $x$).
Observe that the $\Omega_i(\rho)$ are sets of finite perimeter (see for instance \cite[4.5.11]{FedererBOOK}). If we denote, as usual,
by $\partial^* \Omega_i (\rho)$ their reduced boundaries (in $B_{x, \rho}$), we know that $\partial^* \Omega_i (\rho) \subset K$.
Moreover:
\begin{itemize}
\item[(a)] by the rectifiability of the reduced boundary (cf. \cite[4.5.6]{FedererBOOK}), $\partial^* \Omega_i (\rho) \subset \mathcal{R}$;
\item[(b)] each point $x\in \partial^* \Omega_i (\rho)$ belongs to at most another distinct $\partial^* \Omega_j (\rho)$, because at any point $y\in \partial^*\Omega$ of a Caccioppoli set $\Omega$ its blow-up is a half-space, cf. \cite[4.5.5]{FedererBOOK}.
\end{itemize}

Since in what follows we will often deal with subsets of the sphere $\partial B_\rho$, we will use the following notation:
\begin{itemize}
\item $\partial_{\partial B_\rho} A$ is the topological boundary of $A$ as subset of $\partial B_\rho$;
\item $\partial^*_{\partial B_\rho} A$ is the reduced boundary of $A$ relative to $\partial B_\rho$.
\end{itemize}

Using the slicing theory for sets of finite perimeter we can infer that 
\begin{equation}\label{e:(c)}
\mathcal{H}^{n-1} (\partial_{\pa B_t}^*(\Omega_i(\rho)\cap \pa B_t )\setminus ((\partial^* \Omega_i (\rho))\cap \partial B_t)) =0
\qquad \mbox{for a.e. $t< \rho$.}
\end{equation}
This can be for instance proved identifying $\Omega_i (\rho)$ and $\partial^* \Omega_i (\rho)$ with the corresponding integer rectifiable
currents (see \cite[Remark 27.7]{SimonLN}) and then using the slicing theory for integer rectifiable currents (cf. \cite[Chapter 6, Section 28]{SimonLN}).
Combining (a), (b) and \eqref{e:(c)} above we eventually achieve
\begin{equation}\label{slicingmisure}
 \sum_i\H^{n-1}\res\partial_{\pa B_t}^*(\Omega_i(\rho)\cap \pa B_t ) \leq 2\H^{n-1}\res (\mathcal R\cap \pa B_t)\quad
 \mbox{for a.e. }t<\rho.
\end{equation}

\medskip

\textit{Step 1.} In this first step we show that, for every $\e_0>0$ and every $r_0>0$ small enough, there exists $\rho\in ]r_0,2r_0[$ satisfying
\begin{equation}\label{quasi tutto}
 \max_i\big\{\H^n(\Gamma_i(\rho))\big\} \geq (\sigma_n - \e_0)\rho^n.
\end{equation}
Indeed, by \eqref{nulldens}, we consider $r_0$ so small that $\H^n(\mathcal R\cap B_s(x))\leq \e_0 s^n$ for every $s \leq 2r_0$.  
We first claim the existence of a closed set $R \subset ]r_0,2r_0[$ of positive measure such that the following holds $\forall \rho \in R$:
\begin{itemize}
\item [(i)] \[
  \lim_{\sigma\in R, \sigma\rightarrow\rho} \H^{n-1}(\mathcal R\cap \partial B_\sigma)= \H^{n-1}(\mathcal R\cap \partial B_\rho);
 \]
 \item [(ii)] $\H^n(K\cap\partial B_\rho)=0$;
\item [(iii)] $\H^{n-1}(\mathcal R\cap \partial B_\rho)\leq C\e_0\rho^{n-1}$.
\end{itemize}
The existence of a set of positive measure $R'$ such that (iii) holds at any $\rho\in R'$ is an obvious consequence of the coarea formula and of Chebycheff's inequality, provided the universal constant $C$ is larger than $2^n$. Moreover, condition (ii) holds at all but countable many radii. Next, since the map $t\mapsto \H^{n-1}(\mathcal R\cap \partial B_t)$ is measurable, by Lusin's theorem we can select a closed subset $R$ of $R'$ with positive measure for which (i) holds at every radius. 

Fix now a point $\rho\in R$ of density $1$ for $R$: it turns out that $\rho$ satisfies indeed condition \eqref{quasi tutto}. In order to show that estimate, we first choose $(\rho_k)\subset R$, $\rho_k\uparrow \rho$ such that \eqref{slicingmisure} holds for $t=\rho_k$. 
Observe that, for every sequence of points $x_k\in\partial B_{\rho_k}\cap \Omega_i(\rho)$ 
converging to some $x_\infty$, we have that $x_\infty \in \Gamma_i(\rho)\cup (K \cap \pa B_\rho)$, otherwise there would exist $\tau>0$ such that $B_\tau(x_\infty)\cap\Omega_i(\rho)=\emptyset$, against the convergence of $x_k$ to $x_\infty$.  
In particular, rescaling everything at radius $\rho$, for every $\eta>0$ there exists $k(\eta)$ such that, for all $k\geq k(\eta)$ 
\[
 E_{k,i} := \frac{\rho}{\rho_k}\left ( \Omega_i(\rho)\cap \pa B_{\rho_k}\right )\subset U_\eta\left (\Gamma_i(\rho)\cup (K \cap\pa B_\rho)\right )\cap \pa B_\rho =:\Gamma_{i,\eta},
\]
where $U_\eta$ denotes the $\eta$-tubular neighborhood.  
 
 Observe that $\H^n(\Gamma_{i,\eta})\downarrow \H^n(\Gamma_i(\rho))$ as $\eta\downarrow 0$, because 
$\partial_{\pa B_\rho}\Gamma_i(\rho)\subset K \cap \pa B_\rho$ and (ii) holds. 
On the other hand, for every $\eta>0$, we can take $\Lambda_\eta$ compact subset of $\Gamma_i(\rho)$ with $\H^n(\Gamma_i(\rho)\setminus \Lambda_\eta)<\eta$ and $U_\alpha(\Lambda_\eta)\cap B_\rho \subset \Omega_i(\rho)$ for some small $\alpha(\eta)>0$. 
Therefore, for $\rho - \rho_k< \alpha(\eta) $, the following holds 
\[
 E_{k,i}\supset \Lambda_\eta.
\]
Since $\Lambda_\eta \subset E_{k,i} \subset \Gamma_{i,\eta}$ for every $k>k(\eta)$, $\Lambda_\eta \subset \Gamma_i(\rho) \subset \Gamma_{i,\eta}$ and 
$\H^n(\Gamma_{i,\eta}\setminus \Lambda_\eta)\downarrow 0$ as $\eta\downarrow 0$,
we easily deduce that $E_{k,i} \rightarrow \Gamma_i(\rho)$ in $L^1(\pa B_\rho)$. 
Moreover, by \eqref{slicingmisure}
\begin{equation}\label{pp}
 \sum_i \H^{n-1}(\partial^*_{\pa B_{\rho_k}}(\Omega_i(\rho)\cap \pa B_{\rho_k})) \leq 2 \H^{n-1} (\mathcal R \cap \pa B_{\rho_k}).
\end{equation}
The $L^1$ convergence shown above, the lower semicontinuity of the perimeter and the definition of $E_{k,i}$ imply that 
\begin{equation*}
\begin{split}
\sum_i \H^{n-1}(\partial^*_{\pa B_{\rho}}\Gamma_i(\rho))&\leq \liminf_k \sum_i \H^{n-1}(\partial^*_{\pa B_{\rho}}E_{k,i})  \\
& =\liminf_k\left(\frac{\rho}{\rho_k}\right)^{n-1} \sum_i \H^{n-1}(\partial^*_{\pa B_{\rho_k}}(\Omega_i(\rho)\cap \pa B_{\rho_k})). 
 \end{split}
 \end{equation*}
Plugging \eqref{pp}, conditions (i) and (iii) in the previous equation, we get
\begin{equation*}
  \sum_i \H^{n-1}(\partial^*_{\pa B_{\rho}}\Gamma_i(\rho)) \leq 2 \liminf_k \H^{n-1} (\mathcal R \cap \pa B_{\rho_k}) =2 \H^{n-1} (\mathcal R \cap \pa B_{\rho})  \leq C\e_0 \rho^{n-1}.
 \end{equation*}
Let us denote by $\Gamma_0(\rho)$ the element of largest $\H^n$ measure among the $\Gamma_i(\rho)$: 
applying the isoperimetric inequality \cite[Lemma 9]{DelGhiMag} in $\pa B_\rho$ we get
\[
 \sigma_n\rho^n - \H^n(\Gamma_0(\rho))=
 \sum_{i\ge 1}\H^n(\Gamma_i(\rho))\leq C \left(\H^{n-1}(\mathcal R\cap \pa B_\rho\right))^{\frac{n}{n-1}}\leq C \e_0^{\frac{n}{n-1}} \rho^n,
\]
namely $\H^n(\Gamma_0(\rho)) \geq ( \sigma_n -C \e_0^{\frac{n}{n-1}}) \rho^n$, which proves \eqref{quasi tutto}. 

\medskip

\textit{Step 2.} In this second step we
let $\Gamma_0^j(\rho)$ be a $\sim_{K_j}$-equivalence class of largest $\H^n$-measure in $\pa B_\rho \setminus K_j$ and
we claim that: 
\begin{equation}\label{al limite}
 \liminf_j \H^n(\Gamma_0^j(\rho)) \geq \H^n(\Gamma_0(\rho)),
\end{equation}
(where, consistently, $\Gamma_0 (\rho)$ is a $\sim_K$-equivalence class of largest measure in $\pa B_\rho\setminus K$; note that
the latter estimate, combined with Step 1, implies, for $\varepsilon_0$ sufficiently small and $j$ sufficiently large, that such equivalence classes of largest $\H^n$ measure are indeed uniquely determined).

Recall that $\Omega_0(\rho)$ is associated to $\Gamma_0(\rho)$ according to Definition \ref{def good class}. 
Let us consider $\delta>0$ sufficiently small and $\bar\Gamma\subset\subset \Gamma_0(\rho)$ verifying
\begin{equation}\label{misura interna}
\H^n(\bar\Gamma)\geq \H^n(\Gamma_0(\rho)) -\delta\, .
\end{equation}
Next, by compactness, we can uniformly separate $\bar\Gamma$ and $K$, that is we can pick $\eta>0$ sufficiently small so that
\begin{equation}\label{eq:tildegamma}
 V := \bigcup_{s\in [\rho-\eta,\rho]} \frac{s}{\rho}\bar\Gamma = 
 \Big\{x\in \overline{B_\rho}\setminus B_{\rho-\eta} : \rho \frac{x}{|x|}\in \bar\Gamma\Big\}\subset \subset \overline{B}_\rho\setminus K\, . 
\end{equation}

Next we choose an open connected subset of $\Omega_0(\rho)$ with smooth boundary, denoted by $\Omega(\rho)$, such that 
\begin{equation}\label{eq:stima1}
|\Omega_0(\rho)\setminus \Omega(\rho)| <\delta\eta . 
\end{equation}
The set $\Omega(\rho)$ can be constructed as follows:
\begin{itemize}
 \item first one considers $\Lambda \subset\subset \Omega_0(\rho)$ compact with $|\Omega_0(\rho)\setminus \Lambda| <\delta\eta$: 
 this can be achieved for instance looking at a Whitney subdivision of $\Omega_0(\rho)$, taking the union of the cubes with side length bounded from below by a small number;
 \item $\Lambda$ can be enlarged to become connected by adding, if needed, a finite number of arcs at positive distance from $\partial \Omega_0(\rho)$;
\item we can finally take a $C^\infty$ function $f:\R^{n+1}\rightarrow \R$ such that $f|_{\R^{n+1}\setminus \Omega_0(\rho)} = 0$, $f|_\Lambda = 1$ and $0\leq f \leq 1$: 
by the Morse-Sard Theorem, one can choose $t \in ]0,1[$ such that $\{f=t\}$ is a $C^\infty$ submanifold.
\item the connected component of $\{f>t\}$ containing $\Lambda$ satisfies the required assumptions.
\end{itemize}
Since $\Omega(\rho) \subset\subset \Omega_0(\rho)$, by weak convergence $\H^{n}(K_j\cap \Omega(\rho)) \rightarrow 0$; moreover 
since $\Omega(\rho)$ is smooth and connected, it satisfies the isoperimetric inequality 
\begin{equation}\label{eq:invasion}
 |\Omega(\rho) \setminus \Omega^j(\rho) | \leq {\rm Iso}(\Omega (\rho)) \H^n(\Omega(\rho)\cap K_j)^{\frac{n+1}{n}} = o_j (1) \, .
\end{equation}
where $\Omega^j(\rho)$ is the connected component of $\Omega(\rho)\setminus K_j$ of largest volume and ${\rm Iso}(\Omega (\rho))$ is the isoperimetric constant of the smooth connected domain $\Omega(\rho)$ (for the isoperimetric inequality see \cite[Theorem 4.5.2(2)]{FedererBOOK} and use
the fact that $\partial^* \Omega^j (\rho)\subset K_j$, which has been observed above). 
 
Obviously \eqref{eq:tildegamma} implies $\H^n(K_j\cap V)\rightarrow 0$ and, by projecting $K_j\cap V$ on 
$\pa B_\rho$ via the radial map $\Pi: B_\rho \ni x \mapsto \tfrac{\rho}{|x|}x\in \pa B_\rho$, we easily get that the set
$$
 \bar\Gamma^j :=\{y\in\bar\Gamma : \Pi^{-1}(y)\cap V \cap K_j = \emptyset \}
$$
verifies
\begin{equation}\label{q}
 \begin{split} 
 \H^n(\bar\Gamma^j)&=\H^n(\bar\Gamma)-\H^n(\Pi(  K_j \cap V)) \\
 & \geq\H^n(\bar\Gamma) - (\Lip\Pi|_{B_\rho\setminus B_{\rho-\eta}})^n \H^n(K_j\cap V) = \H^n(\bar\Gamma) -  o_j(1).
 \end{split}
\end{equation}
We deduce from \eqref{q} that
\begin{equation}\label{eq:V}
|V \cap \Pi^{-1}(\bar\Gamma^j)| =|V|-|V\setminus \Pi^{-1}(\bar\Gamma^j)| \geq |V| - \eta \H^n(\bar\Gamma\setminus\bar\Gamma^j) \geq |V| - o_j(1). 
\end{equation}
The previous inequality, \eqref{eq:tildegamma}, \eqref{eq:stima1} and \eqref{eq:invasion} in turn imply that
\begin{equation}\label{almost V}
\begin{split}
|V\setminus (\Omega^j(\rho)\cap \Pi^{-1}(\bar\Gamma^j))| & \quad \, \, \,\, \leq \quad \, \, \,\, |V\setminus \Omega^j(\rho)|+ |V\setminus \Pi^{-1}(\bar\Gamma^j)|   \\
 &   \overset{\eqref{eq:tildegamma}, \eqref{eq:V}}{\leq}   |\Omega_0(\rho)\setminus \Omega(\rho)|+ |\Omega(\rho)\setminus \Omega^j(\rho)|+ o_j(1) \\
 &\overset{\eqref{eq:stima1}, \eqref{eq:invasion}}{\leq}   \eta \delta +   o_j(1).
\end{split}
\end{equation}
If $x \in V\cap\Omega^j(\rho)   \cap \Pi^{-1}(\bar\Gamma^j)$,  then $x\sim_{K_j}\Pi(x)$ 
using as a path simply the radial segment $[x,\Pi(x)[$; moreover we can always connect two points belonging to $V\cap\Omega^j(\rho)\cap \Pi^{-1}(\bar\Gamma^j)$ with a path inside $\Omega^j(\rho)$. 
But, by \eqref{almost V}, the endpoints $\Pi(x)$ of these segments must cover all but a small fraction $G_j$ of  $\bar\Gamma^j$ of measure $o_j(1)$. 
Indeed we can estimate the complement set $G_j :=\Pi(V \setminus (\Omega^j(\rho)\cap \Pi^{-1}(\bar\Gamma^j)) )$ using the coarea formula and the 
self similarity of the shells:
\begin{equation*}
\begin{split}\frac{\rho}{n+1}  \left(1 - \left (1-\frac{\eta}{\rho}\right)^{n+1}\right)\H^n(G_j) &= \int_{\rho-\eta}^\rho \H^n\left (\frac{t}\rho G_j\right )dt \\
&\leq  |V\setminus (\Omega^j(\rho)\cap \Pi^{-1}(\bar\Gamma^j))| \leq  \delta\eta +   o_j(1),
\end{split}
\end{equation*}
which yields, for $\eta$ small enough (namely smaller than a positive constant $\eta_0 (n, \rho)$), 
\begin{equation}\label{eq:sfinimento}
\H^n(G_j) \leq 2 \delta + o_j(1)\, . 
\end{equation}
By concatenating the paths we conclude that $\bar\Gamma^j\setminus G_j$ must be contained in a unique equivalence class $\Gamma^j_i(\rho)$. 
We remark that for the moment we do not know whether $\Gamma_i^j(\rho)$ is an equivalence class of $\pa B_\rho \setminus K_j$ with largest measure. Summarizing the inequalities achieved so far we conclude
\begin{equation*}
\begin{split}
\H^n (\Gamma_0^j (\rho)) \geq \H^n(\Gamma_i^j(\rho))  & \, \, \,  \geq \, \, \, \, \H^n(\bar\Gamma^j\setminus G_j)  \overset{\eqref{eq:sfinimento}}{\geq} 
\H^n(\bar\Gamma^j)-2\delta - o_j(1)  \\
&\overset{\eqref{q}}{\geq} \H^n(\bar\Gamma) - o_j(1) - 2\delta \overset{\eqref{misura interna}}{\geq}
\H^n(\Gamma_0(\rho)) - o_j(1) - 3\delta 
\end{split}
\end{equation*}
In particular, letting first $j\uparrow\infty$ and then $\delta\downarrow 0$ we achieve \eqref{al limite}.
 
\medskip

\textit{Step 3.} We recover a straightforward contradiction since, by the density lower bound (d.l.b.) proven in Lemma \ref{l:density}, the good class property (g.c.p.) of $\P(H)$, the lower semicontinuity (l.s.) in the weak convergence \eqref{muj va a mu} and the bound \eqref{cost per area}, we get
 \begin{equation*}
 \begin{split}
 c_0\rho^n &\overset{d.l.b.}{\leq} \mu(B_{x,\rho}) \overset{l.s.}{\leq} \liminf_j \FF(K_j,B_{x,\rho})\\
  &\overset{g.c.p.}{\leq} \liminf_j \FF(\pa B_{x,\rho}\setminus \Gamma_0^j(\rho)) \overset{\eqref{cost per area}}{\leq} \Lambda \liminf_j \H^n(\pa B_{x,\rho}\setminus \Gamma_0^j(\rho)).
  \end{split}
  \end{equation*}
Plugging in \eqref{al limite} and \eqref{quasi tutto} (both relative to the complementary sets), we get 
 \begin{equation*}
 \begin{split}
 c_0\rho^n &\overset{\eqref{al limite}}{\leq} \Lambda \H^n(\pa B_{x,\rho}\setminus \Gamma_0(\rho))\overset{\eqref{quasi tutto}}{\leq} \Lambda \e_0 \rho^n,
  \end{split}
  \end{equation*}
 which is false for $\e_0$ small enough. We conclude $\H^n(\mathcal N)=0$, hence the rectifiability of the set $K$.

\subsection{Proof of Theorem \ref{thm generale}: semicontinuity} We are now ready to complete the proof of Theorem \ref{thm generale}, namely to show $\liminf_j\FF(K_j)\ge \FF(K)$ and \(\mu=F(x,T_K(x))\H^n \res K\), when the integrand $F$ is elliptic.\\
\noindent
We claim indeed that $\theta(x) \geq F(x,T_K(x))$ for every $x$ where the rectifiable set $K$ has an approximate tangent plane 
$\pi = T_K (x)$ and $\theta$ is approximately continuous. Let $x$ be such point and assume, without loss of generality, that $x=0$. 
We therefore have the following limit in the weak$^{*}$ topology:
\begin{equation}\label{eq:full}
\theta(r\cdot) \H^n \res \frac Kr \weak \theta(0) \H^n\res \pi\qquad \mbox{ for }\qquad r\downarrow 0.
\end{equation}

For a suitably chosen sequence $r_j\downarrow 0$, consider the corresponding rescaled sets $\tilde{K}_j := \frac{1}{r_j} K_j$ and rescaled measures
 $\tilde{\mu}_j :=\tilde{F}_j \mathcal{H}^n \res \tilde{K}_j$, where $\tilde{F}_j (y) := F (r_j y, T_{\tilde{K}_j} (y))$. 
With a diagonal argument, if $r_j\downarrow 0$ sufficiently slow (since the blow-up to $\pi$ in \eqref{eq:full} happens on the full continuous limit $r\downarrow 0$), then we can assume that the $\tilde{\mu}_j$ are converging weakly$^*$ in $\mathbb R^{n+1}$ to $\tilde \mu= \theta (0) \mathcal{H}^n \res \pi$.
Note moreover that $\tilde{\mu}_j (B_1) \to \omega_n \theta (0)$ because $\tilde{\mu} (\partial B_1 )=0$.

Let $\tilde{\Omega}_j$ be the largest connected component of $B_1 \setminus \tilde{K}_j$. As already observed, $\tilde{\Omega}_j$ is a Caccioppoli set and $\partial^* \tilde{\Omega}_j \subset \tilde{K}_j$. Up to subsequences, we can assume that $\tilde{\Omega}_j$ converges as a Caccioppoli set to some $\tilde{\Omega}$ whose reduced boundary in $B_1$ must be contained in $\pi$. We thus have three alternatives:
\begin{itemize}
\item[(i)] $\tilde{\Omega}$ is the lower or the upper half ball of $B_1\setminus \pi$. In this case, the lower semicontinuity of the energy $\mathbf{F}$ on Caccioppoli sets (which follows from \cite[5.1.2 \& 5.1.5]{FedererBOOK}) implies
\begin{align*}
\omega_n r^n F (0, \pi) \leq &\liminf_{j\to \infty} \int_{\partial^* \Omega_j\cap B_1} F (0, T_{\partial^* \Omega_j} (y)) d\mathcal{H}^n (y)\nonumber\\
= &\liminf_{j\to \infty} \int_{\partial^* \Omega_j\cap B_1} F (r_j y, T_{\partial^* \Omega_j} (y)) d\mathcal{H}^n (y)\nonumber\\
\leq &\lim_{j\to \infty} \int_{\tilde{K}_j \cap B_1} \tilde{F} (y) \, d\mathcal{H}^n (y) = \lim_{j\to \infty}\tilde{\mu}_j (B_1) = \omega_n \theta (0)\, ,
\end{align*}
which is the desired inequality. 
\item[(ii)] $\tilde{\Omega}$ is the whole $B_1$;
\item[(iii)] $\tilde{\Omega}$ is the empty set. 
\end{itemize}
The third alternative is easy to exclude. Indeed in such case $|\tilde{\Omega}_j|$ converges to $0$. On the other hand, if we consider one of the two connected components of $B_1 \setminus U_{1/100} (\pi)$, say $A$, we know that $\mathcal{H}^n (\tilde K_j \cap A)$ converges to $0$ (since $\tilde{\mu}_j\rightharpoonup^* \theta (0) \mathcal{H}^n \res \pi$). The relative isoperimetric inequality implies that the volume of the largest connected component of $A\setminus \tilde K_j$ converges to the volume of $A$ (cf. the argument for \eqref{eq:invasion}). 

Consider next alternative (ii). We argue similarly to step 2 of the previous subsection. Consider a fixed $\varepsilon> 0$ and set $\bar \Gamma^+ =
(\partial B_1)^+\setminus U_{3\varepsilon} (\pi)$, where $(\partial B_1)^+ = \partial B_1 \cap \{x_{n+1}>0\}$, having set $x_{n+1}$ a coodinate direction orthogonal to $\pi$: in particular $\H^n(\bar\Gamma^+) \geq \sigma_n/2 -C\e$. 

Similary to step 2 consider 
\[
V = \bigcup_ {1-\varepsilon\leq s\leq 1} s \bar \Gamma^+\, 
\]
consisting of the segments $S_x := [(1-\varepsilon) x, x]$ for every $x\in \bar \Gamma^+$. In particular for $\varepsilon$ sufficiently small we have 
$V \subset B_1 \setminus U_{2\varepsilon} (\pi)$ and thus we know that
\begin{itemize}
\item[(a)] $\mathcal{H}^n (\tilde{K}_j \cap V) \to 0$;
\item[(b)] $|V\setminus \tilde{\Omega}_j|\to 0$. 
\end{itemize}
In particular, if we consider, as in step 3 above, the set $\tilde{G}_j^+\subset \bar{\Gamma}^+$ of points for which either $S_x \cap \tilde \Omega_j = \emptyset$ or $S_x \cap \tilde{K}_j \neq \emptyset$, we conclude that $\mathcal{H}^n (\tilde{G}^+_j) = o_j (1)$.

If we define symmetrically the sets $\bar{\Gamma}^-$ and $\tilde{G}^-_j$, the same argument gives us $\H^n(\bar\Gamma^-) \geq \sigma_n/2 -C\e$ 
as well as $\mathcal{H}^n (\tilde{G}^-_j) = o_j (1)$. Choosing now $\varepsilon$ small and an appropriate diagonal sequence, we conclude the existence of a sequence of sets  
$\tilde{\Gamma}_j =\bar{\Gamma}^+\cup\bar{\Gamma}^- \setminus(\tilde{G}^+_j\cup \tilde{G}^-_j) \subset \partial B_1 \setminus \tilde{K}_j$ 
with the property that
\begin{itemize}
\item $\mathcal{H}^n (\tilde{\Gamma}_j)$ converges to $\sigma_n$;
\item any two points $x,y\in \tilde{\Gamma}_j$ can be connected in $\bar B_1 (0)$ with an arc which does not intersect $\tilde{K}_j$. 
\end{itemize}
Therefore each $\tilde{\Gamma}_j$ must be contained in a unique equivalence class $\Gamma_{i(j)}(\tilde K_j,0,1)$.
Coming to the sets $K_j$ by scaling backwards, we find a sequence of sets $\Gamma_{i(j)} (K_j, 0, r_j)$ in the equivalence classes of Definition \ref{def:cup}
such that
\[
\lim_{j\to\infty} \frac{\mathcal{H}^n (\partial B_{0, r_j} \setminus \Gamma_{i(j)}(K_j, 0, r_j))}{r_j^n}=0\, .
\]
Considering the bound $\mathbf{F} (K_j \cap B_{r_j}) \leq \Lambda \mathcal{H}^n (\partial B_{0, r_j} \setminus \Gamma_{i(j)} (K_j, 0, r_j))$, we can pass to the rescaled measures again to conclude $\tilde{\mu}_j (B_1 (0)) = o_j (1)$, clearly contradicting the assumption that $\mu_j (B_1 (0))$ converges to the positive number $\theta (0)$.

\subsection{Proof of Proposition \ref{p:aggiuntiva}}
We first need the following estimates. 

\begin{lemma}\label{l:lista}
Let $K$ be the $n$-rectifiable set obtained in the previous section. For every $x$ where $K$ has an approximate tangent plane $T_K (x)$, let $O_x$ be the special orthogonal transformation of $\mathbb R^{n+1}$ mapping $\{x_1=0\}$ onto $T_K (x)$ and set
$\bar{Q}_{x,r} = O_x (Q_{x,r})$ and $\bar R_{x,r,\e r} = O_x (R_{x, r \e r})$. 

Then at $\mathcal{H}^n$ almost every $x\in K$ the following holds: for every $\e>0$ there exist $r_0=r_0(x) \leq\tfrac{1}{\sqrt{n+1}} \dist(x,H)$ such that, for $r\leq r_0 /2$, 
\begin{gather}
(\lambda \omega_n-\e)r^n\leq \mu(B_{x,r})\le (\theta(x)\omega_n+\e)r^n,\qquad
(\theta(x)-\e)r^n< \mu(\bar Q_{x,r})< (\theta(x)+\e)r^n, \label{uno}\\
\sup_{y\in B_{x,r_0}, \,S\in G(n+1,n)}|F(y,S) - F(x,S)|\leq \e.\label{tre}
\end{gather}
Moreover, for almost every such $r$, there exists  $j_0(r) \in \N$ such that for every $j\geq j_0$:
\begin{gather}
(\theta(x)\omega_n-\e)r^n \le\FF(K_j,B_{x,r})\leq (\theta(x) \omega_n + \e ) r^n,\label{due}\\
(\theta(x)-\e)r^n\leq \FF(K_j,\bar Q_{x,r})\leq (\theta(x)+\e)r^n, \qquad \FF(K_j,Q_{x,r}\setminus \bar R_{x,r,\e r})<\e r^n \label{quattro}.
\end{gather}
\end{lemma}
\begin{proof}
Fix a point $x$ where $T_K(x)$ exists and $\theta$ is approximately continuous: for the sake of simplicity, we can assume $x=0$ and that, after a rotation, the approximate tangent space at $0$ coincides with $T_K=\{x^{n+1}=0\}$. 
For  almost every $r\leq r_0/2$ we can suppose that $\mu(\pa B_r)=\mu(\pa Q_{r})=\mu(\pa R_{r,\e r})=0$; 
moreover by  rectifiability and the density lower bound \eqref{lower density estimate mu}, we also know that 
$B_r \cap K \subset U_{\e r}(T_K) $ (see the proof above). The second equation in \eqref{quattro} follows than by weak convergence.  We also know that, up to further reducing $r_0$,  for $r\leq r_0/2$, \eqref{3} and \cite[Theorem 2.83]{AFP} imply that
\begin{gather}
(\theta(0)\omega_n-\e)r^n < \mu(B_r)< (\theta(0)\omega_n+\e)r^n\label{2},\\
	(\theta(0)-\e)r^n< \mu(Q_r)< (\theta(0)+\e)r^n    \label{6}.
\end{gather}
Again by weak convergence, we recover \eqref{due} and the first equation in \eqref{quattro}. 
Moreover \eqref{tre} is a consequence of \eqref{oscill}. 
Finally \eqref{uno} follows from \eqref{2}, \eqref{6} and $\theta\geq \lambda$, whereas the latter bound is a consequence of the previous subsection where we have shown
$\theta (x)\geq F (x, T_K (x))$.
\end{proof}

 Assume w.l.o.g. $x=0$. 
 Arguing by contradiction, we assume that $\theta(0)=F(0,T_K(0))+\sigma$ for some  $\sigma>0$ and let $\e<\min\{\tfrac{\sigma}{2},\frac{\lambda\s}{ 4\Lambda}\}$. 
As a consequence of \eqref{quattro}, there exist $r$ and $j_0=j_0(r)$ such that
  \begin{equation}
    \label{buonpunto}
      \FF(K_j,Q_{r})>\Big(F(0,T)+\frac{\s}2\Big)\,r^n\,,\qquad \FF(K_j,Q_{r}\setminus R_{r,\e r})<\frac{\lambda\s}{ 4\Lambda}\,r^n,\qquad\forall j\ge j_0\,.
  \end{equation}
 Consider the map $P\in \D(0,r)$  defined in \cite[Equation 3.14]{DePDeRGhi} which collapses \(R_{r(1-\sqrt{\e}),\e r}\) onto the tangent plane $T_K$ and satisfies  $\|P - Id\|_\infty +\Lip (P - Id)\le C\,\sqrt\e$. Exploiting the fact that \(\mathcal P(H)\) is a deformation class and by almost minimality of $K_j$, we find that
  \begin{equation*}
  \begin{split}
    \FF(K_j,Q_{r})-o_j(1) &\le 
    \underbrace{\FF( P(K_j),P(R_{(1-\sqrt{\e})r,\e r}))}_{I_1}
    +
    \underbrace{\FF( P(K_j),P(R_{r,\e r}\setminus R_{(1-\sqrt{\e})r,\e r}))}_{I_2} \\
    &+
    \underbrace{\FF( P(K_j),P(Q_{r}\setminus R_{r,\e r}))}_{I_3}\,.
    \end{split}
  \end{equation*}
  By the properties of $P$ and \eqref{tre}, we get $I_1\le (F(0,T_K)+\e) r^n$, while, by \eqref{buonpunto} and equation \eqref{cost per area}
  \[
  I_3\leq \frac{\Lambda}{\lambda}\, (\Lip P)^n\,  \FF(K_j,Q_{r}\setminus R_{r,\e r})<(1+C\,\sqrt\e)^n\,\frac{\s}{4 }\,r^n\,.
  \]
  Since $\FF( P(K_j),P(R_{r,\e r}\setminus R_{(1-\sqrt{\e})r,\e r}))\leq \frac{\Lambda}{\lambda}(1+C\sqrt{\e})^n\FF(K_j,R_{r,\e r}\setminus R_{(1-\sqrt{\e})r,\e r})$ and $ R_{r,\e r}\setminus R_{(1-\sqrt{\e})r,\e r}\subset Q_{(1-\sqrt\e)r}\setminus  Q_{r}$, by \eqref{quattro} we can also bound
 \begin{multline*}
  I_2=\FF( P(K_j),P(R_{r,\e r}\setminus R_{(1-\sqrt{\e})r,\e r}))  \le\frac{\Lambda}{\lambda}(1+C\sqrt{\e})^n\FF(K_j,Q_{r}\setminus Q_{(1-\sqrt\e)r})
    \\
    \le C(1+C\sqrt{\e})^n\Big((F(0,T_K)+\s+\e)-(F(0,T_K)+\s-\e)(1-\sqrt\e)^n\Big)\,r^n 
    \le C\sqrt{\e}r^n.
\end{multline*}

Hence, as $j\to\infty$, by \eqref{uno}
  \[
  \Big(F(0,T_K)+\frac{\s}2\Big) r^n\le (F(0,T_K)+\e) r^n +C\sqrt{\e} r^n+(1+C\,\sqrt\e)^n\,\frac{\s}{4}\,r^n\,:
  \]
 dividing by $r^n$ and letting $\e\downarrow 0$ provides the desired contradiction. 

 We obtain that $\theta(x)\leq F(x,T_K(x))$ almost everywhere and, together with the previous step, \(\mu=F(x,T_K(x))\H^n \res K\).

\section{Proof of Theorem \ref{thm plateau}}\label{s4}

Most of the proof of Theorem \ref{thm plateau} relies on the following elementary geometric remark.

\begin{lemma}\label{lemma curve intersezione}
   If $K\in\F(H,\CC)$, $B_{x,r}\cc\R^{n+1}\setminus H$, and $\g\in\CC$,
   then either $\g\cap (K\setminus B_{x,r})\ne\emptyset$, or there exists a connected component $\s$ of
   $\g \cap \overline{B_{x,r}}$ which is homeomorphic to an interval and whose end-points belong to two distinct equivalence classes $\Gamma_i(x,r)$'s of $\pa B_{x,r}\setminus K$ in the sense of Definition \ref{def good class}. 
\end{lemma}

\begin{proof}
  [Proof of Lemma \ref{lemma curve intersezione}] The proof is analogous to \cite[Lemma 10]{DelGhiMag}, we briefly sketch the argument to  highlight the main differences. Assuming that $\g$ and $\pa B_{x,r}$ intersect transversally (see Step 2 in the proof of \cite[Lemma 10]{DelGhiMag} for the reduction to this case), we can find finitely many mutually disjoint closed arcs $I_i\subset S^1$, $I_i=[a_i,b_i]$, such that $\g\cap B_{x,r}=\bigcup_i\g((a_i,b_i))$ and $\g\cap\pa B_{x,r}=\bigcup_i\{\g(a_i),\g(b_i)\}$. Arguing by contradiction we may assume that for every $i$ there exists an equivalence class $\Gamma_i(x,r)$ of $\pa B_{x,r}\setminus K$ such that $\g(a_i),\g(b_i)\in \Gamma_i(x,r)$.  By connectedness of the associated $\Omega_i(x,r)$ (see the discussion after Definition \ref{def good class}) and the definition of $\Gamma_i(x,r)$, for each $i$ we can find a smooth embedding $\tau_i:I_i\rightarrow \Omega_i(x,r) \cup \Gamma_i(x,r)$ such that $\tau_i(a_i)=\g(a_i)$ and $\tau_i(b_i)=\g(b_i)$; moreover since $n\geq 2$, one can easily achieve this by enforcing $\tau_i(I_i)\cap\tau_j(I_j)=\emptyset$. Finally, we define $\bar\g$ by setting $\bar\g=\g$ on $S^1\setminus\bigcup_iI_i$, and $\bar\g=\tau_i$ on $I_i$.  In this way, $[\bar\g]=[\g]$ in
$\pi_1(\R^{n+1}\setminus H)$, with $\bar\g\cap K\setminus\overline{B_{x,r}}=\g\cap K\setminus\overline{B_{x,r}}=\emptyset$ and $\bar\g\cap K\cap\overline{B_{x,r}}=\emptyset$ by construction; that is, $\bar\g\cap K=\emptyset$. Since there exists $\widetilde\g\in\CC_H$ with $[\widetilde\g]=[\bar\g]=[\g]$ in $\pi_1(\R^{n+1}\setminus H)$ which is uniformly close to $\bar\g$, we infer $\widetilde\g\cap K=\emptyset$, and thus find a contradiction to $K\in\F(H,\CC)$.
\end{proof}

\begin{proof}
  [Proof of Theorem \ref{thm plateau}] (a): We start showing that $\F(H,\CC)$ is a good class in the sense of Definition \ref{def good class}.  To this end, we fix $V\in\F (H,\CC)$ and $x\in V$, and prove that a.e. $r\in (0, \dist (x, H))$ one has $V'\in\F(H,\CC)$, where $V'$ is  a cup competitor of $V$ in $B_{x,r}$. We thus fix $\g\in\CC$ and, without loss of generality, we assume that $\g\cap(V\setminus B_{x,r})=\emptyset$. By Lemma \ref{lemma curve intersezione}, $\gamma$ has an arc contained in $\overline{B_{x,r}}$ homeomorphic to $[0,1]$ and whose end-points belong to distinct equivalence classes of $\pa B_{x,r}\setminus V$; we denote by $\sigma :[0,1]\rightarrow \overline{B_{x,r}}$ a parametrization of this arc. Since $V'$ must contain all but one $\Gamma_i(x,r)$, either $\s(0)$ or $\s(1)$  belongs to $\g\cap V'\cap \pa B_{x,r}$. 
  
  \medskip

  \noindent (b):  The second statement of the Theorem, namely that $K\in\F(H,\CC)$, has exactly the same proof as in \cite[Theorem 4(b)]{DelGhiMag}. It follows that all the results of Theorem \ref{thm generale} apply.

  \medskip

  \noindent (c): We observe that $K$ is a $(\FF,0,\infty)$-minimal set, i.e.
  \[
  \FF(K)\le \FF(\vphi(K))
  \]
  whenever $\vphi:\R^{n+1}\to\R^{n+1}$ is a Lipschitz map such that $\vphi=\Id$ on $\R^{n+1}\setminus B_{x,r}$ and $\vphi(B_{x,r})\subset B_{x,r}$ for some $x\in\R^{n+1}\setminus H$ and $r<\dist(x,H)$. 

The above inequality is a consequence of $\vphi(K)\in\F (H,\CC)$, which can be proved via degree theory as in \cite[Theorem 4, Step 3]{DelGhiMag}.
\end{proof}


\begin{thebibliography}{DPDRG16}
\ifx \showCODEN  \undefined \def \showCODEN #1{CODEN #1}  \fi
\ifx \showISBN   \undefined \def \showISBN  #1{ISBN #1}   \fi
\ifx \showISSN   \undefined \def \showISSN  #1{ISSN #1}   \fi
\ifx \showLCCN   \undefined \def \showLCCN  #1{LCCN #1}   \fi
\ifx \showPRICE  \undefined \def \showPRICE #1{#1}        \fi
\ifx \showURL    \undefined \def \showURL {URL }          \fi
\ifx \path       \undefined \input path.sty               \fi
\ifx \ifshowURL \undefined
     \newif \ifshowURL
     \showURLtrue
\fi

\bibitem[AFP00]{AFP}
L.~Ambrosio, N.~Fusco, and D.~Pallara.
\newblock {\em {Functions of bounded variation and free discontinuity
  problems}}.
\newblock {Oxford Mathematical Monographs}. The Clarendon Press, Oxford
  University Press, New York, 2000.
\newblock xviii+434 pp pp.

\bibitem[All74]{Allardratio}
William~K. Allard.
\newblock A characterization of the area integrand.
\newblock In {\em Symposia {M}athematica, {V}ol. {XIV} ({C}onvegno di {T}eoria
  {G}eometrica dell'{I}ntegrazione e {V}ariet\`a {M}inimali, {INDAM}, {R}ome,
  1973)}, pages 429--444. Academic Press, London, 1974.

\bibitem[Alm68]{Almgren68}
F.~J.~Jr Almgren.
\newblock {Existence and regularity almost everywhere of solutions to elliptic
  variational problems among surfaces of varying topological type and
  singularity structure}.
\newblock {\em Ann. Math.}, 87:\penalty0 321--391, 1968.

\bibitem[Alm76]{Almgren76}
F.~J.~Jr. Almgren.
\newblock {Existence and regularity almost everywhere of solutions to elliptic
  variational problems with constraints}.
\newblock {\em Mem. Amer. Math. Soc.}, 4\penalty0 (165):\penalty0 viii+199 pp,
  1976.

\bibitem[Dav14]{davidplateau}
G.~David.
\newblock {Should we solve {P}lateau's problem again?}
\newblock In C.~Fefferman; A. D. Ionescu; D.H. Phong;~S. Wainger, editor, {\em
  {Advances in Analysis: the legacy of Elias M. Stein}}. Princeton University
  Press, 2014.

\bibitem[{De }54]{DeGiorgiSOFP1}
E.~{De Giorgi}.
\newblock {Su una teoria generale della misura $(r-1)$-dimensionale in uno
  spazio ad $r$-dimensioni}.
\newblock {\em Ann. Mat. Pura Appl. (4)}, 36:\penalty0 191--213, 1954.

\bibitem[{De }08]{DeLellisNOTES}
C.~{De Lellis}.
\newblock {\em {Rectifiable sets, densities and tangent measures}}.
\newblock {Zurich Lectures in Advanced Mathematics}. European Mathematical
  Society, Z{\"u}rich, 2008.
\newblock vi+127 pp.

\bibitem[DGM14]{DelGhiMag}
C.~{De Lellis}, F.~Ghiraldin, and F.~Maggi.
\newblock {A direct approach to Plateau's problem}.
\newblock {\em JEMS}, 2014.

\bibitem[DPDRG15]{DePDeRGhi}
G.~De~Philippis, A.~De~Rosa, and F.~Ghiraldin.
\newblock A direct approach to {P}lateau's problem in any codimension.
\newblock {\em Adv. in Math.}, 288:\penalty0 59--80, January 2015.

\bibitem[DPDRG16]{DePDeRGhiCPAM}
G.~De~Philippis, A.~De~Rosa, and F.~Ghiraldin.
\newblock {Rectifiability of varifolds with locally bounded first variation with respect to anisotropic surface energies}.
\newblock {\em Communications on Pure and Applied Mathematics}, 2016.

\bibitem[DPDRG17]{DePDeRGhi2}
G.~De~Philippis, A.~De~Rosa, and F.~Ghiraldin.
\newblock {Existence results for minimizers of parametric elliptic functionals}.
\newblock 2017.
\newblock {Forthcoming}.

\bibitem[DS00]{DavidSemmes}
G.~David and S.~Semmes.
\newblock {Uniform rectifiability and quasiminimizing sets of arbitrary
  codimension.}
\newblock pages viii+132 pp., 2000.

\bibitem[Fed69]{FedererBOOK}
H.~Federer.
\newblock {\em {Geometric measure theory}}, volume 153 of {\em {Die Grundlehren
  der mathematischen Wissenschaften}}.
\newblock Springer-Verlag New York Inc., New York, 1969.
\newblock xiv+676 pp pp.

\bibitem[{Feu}09]{Feuvrier2009}
V.~{Feuvrier}.
\newblock {Condensation of polyhedric structures onto soap films}.
\newblock {\em ArXiv e-prints}, June 2009.

\bibitem[FF60]{federer60}
H.~Federer and W.~H. Fleming.
\newblock {Normal and integral currents}.
\newblock {\em Ann. Math.}, 72:\penalty0 458--520, 1960.

\bibitem[HP13]{HarrisonPugh14}
J.~Harrison and H.~Pugh.
\newblock {Existence and soap film regularity of solutions to {P}lateau's
  problem}.
\newblock 2013.
\newblock arXiv:1310.0508.

\bibitem[HP15]{HarrisonPugh15}
J.~Harrison and H.~Pugh.
\newblock {Solutions To Lipschitz Variational Problems With Cohomological
  Spanning Conditions}.
\newblock 2015.
\newblock arxiv:1506.01692.

\bibitem[HP16]{HarrisonPugh16}
J.~Harrison and H.~Pugh.
\newblock {General Methods of Elliptic Minimization}.
\newblock 2016.
\newblock Forthcoming.

\bibitem[Mag12]{maggiBOOK}
F.~Maggi.
\newblock {\em {Sets of finite perimeter and geometric variational problems: an
  introduction to Geometric Measure Theory}}, volume 135 of {\em {Cambridge
  Studies in Advanced Mathematics}}.
\newblock Cambridge University Press, 2012.

\bibitem[Mat95]{mattila}
P.~Mattila.
\newblock {\em {Geometry of sets and measures in Euclidean spaces. Fractals and
  rectifiability}}, volume~44 of {\em {Cambridge Studies in Advanced
  Mathematics}}.
\newblock Cambridge University Press, Cambridge, 1995.
\newblock xii+343 pp.

\bibitem[Mor66]{Morrey}
C.~B.~Jr Morrey.
\newblock {\em Multiple integrals in the calculus of variations}.
\newblock Springer-Verlag, Berlin Heidelberg New York, 1966.

\bibitem[Pre87]{Preiss}
D.~Preiss.
\newblock {Geometry of measures in {${\bf R}\sp n$}: distribution,
  rectifiability, and densities}.
\newblock {\em Ann. of Math. (2)}, 125\penalty0 (3):\penalty0 537--643, 1987.
\newblock \showISSN{0003-486X}.

\bibitem[Rei60]{reifenberg1}
E.~R. Reifenberg.
\newblock {Solution of the {P}lateau problem for $m$-dimensional surfaces of
  varying topological type}.
\newblock {\em Acta Math.}, 104:\penalty0 1--92, 1960.

\bibitem[Sim83]{SimonLN}
L.~Simon.
\newblock {\em {Lectures on geometric measure theory}}, volume~3 of {\em
  {Proceedings of the Centre for Mathematical Analysis}}.
\newblock Australian National University, Centre for Mathematical Analysis,
  Canberra, 1983.
\newblock vii+272 pp.

\end{thebibliography}
\end{document}